\documentclass[11pt, reqno]{amsart}
\usepackage{amsmath, amsthm, amscd, amsfonts, amssymb, mathtools, color, hyperref}
\usepackage{geometry}
\usepackage[dvipsnames]{xcolor}
\newtheorem{theorem}{Theorem}[section]
\newtheorem{lemma}[theorem]{Lemma}
\newtheorem{proposition}[theorem]{Proposition}
\newtheorem{corollary}[theorem]{Corollary}
\newtheorem{observation}[theorem]{Observation}

\theoremstyle{definition}
\newtheorem{definition}[theorem]{Definition}
\newtheorem{example}[theorem]{Example}
\theoremstyle{remark}

\numberwithin{equation}{section}

\title[Simultaneous triangularization]{Simultaneous triangularization over max-algebras}

\author[Askar]{Askar Ali M$^{1}$}
\address{$^{1}$ Azim Premji University, Bhopal Campus, Bhopal, India}
\email{askar.m@apu.edu.in}

\author[Sachindranath]{Sachindranath Jayaraman$^{2,\ast}$}
\address{$^{2}$ School of Mathematics, IISER Thiruvananthapuram, India}

\email{sachindranathj@iisertvm.ac.in, sachindranathj@gmail.com}
\thanks{$^\ast$Corresponding author}

\author[Himadri]{Himadri Mukherjee$^3$}
\address{$^3$ Department of Mathematics, BITS Pilani K. K. Birla Goa Campus, India}
\email{himadrim@goa.bits-pilani.ac.in}


\begin{document}

\noindent

\begin{abstract}
The purpose of this article is to investigate triangularization and simultaneous triangularization of matrices 
over max algebras using graph theoretic methods. We establish a connection between commutators and commutants 
with simultaneous triangularization over max algebras. We also define the notion of characteristic polynomial 
of a collection in terms of the tropical determinant and determine when it can be written as a product of linear terms. 
Algorithms for all of the above are also brought out.
\end{abstract}

\maketitle

\section{Introduction}

We work throughout over the field $\mathbb{R}$ of real numbers. We concern ourselves with only those matrices whose 
entries are nonnegative real numbers. Other notations and terminologies used in this work will be introduced below. 
By a max algebra, we mean the triple $(\mathbb{R}_{+}, \oplus, \otimes)$, where $\mathbb{R}_{+}$ denotes the set of nonnegative real numbers, $\oplus$ denotes the binary operation of taking the maximum of two nonnegative numbers and $\otimes$ is the usual multiplication of two numbers. There are several abstract examples of max algebras. The one given 
above is more amenable to work with, while dealing with nonnegative matrices. Another example is the set of real numbers, together with $-\infty$, equipped with the binary operations of maximization and addition, respectively. The latter system is isomorphic to the former one via  the exponential map. Max algebras have found numerous applications in several fields 
such as optimization, discrete event dynamical systems, scheduling problems and many more as can be seen from the 
monographs (more details can be found in \cite{Butkovic, Economics-note, Max@work}). The one we work with here 
also has a nice combinatorial appeal, as pointed out in \cite{Butkovic-0}. There are many other references on this topic 
as can be inferred from the MR database. We restrict ourselves to only those needed here for this work.

\medskip
We shall denote by $M_n(\mathbb{R}_+)$ the collection of all $n \times n$ matrices with entries from the max-algebra 
described in the previous paragraph. Given any two such matrices $A$ and $B$, their matrix product, which we denote 
by $AB$ (as in the classical case), is defined by 
\[ [ AB ]_{ij}\ \ =\ \ \max_{k}\ \lbrace a_{ik}\ . \ b_{kj} \rbrace. \]  We shall also denote by $GL_n(\mathbb{R}_+)$ 
the set of all invertible matrices in a max algebra. It is easy to prove that an element of $GL_n(\mathbb{R}_+)$ is necessarily 
a generalized permutation matrix - one that is a product of a diagonal matrix and a permutation matrix. Therefore, 
similarity with respect to invertible matrices in a max algebra follows a special structure. Given an $n \times n$ 
nonnegative matrix $A$, there is a natural way to associate a simple weighted directed graph or simply the digraph 
$G_A$ to the matrix as follows: $G_A$ has $n$ vertices, say $1, \dots, n$, such that there is an edge from $i$ to $j$ with 
weight $a_{ij}$ if and only if $a_{ij} > 0$. By a circuit, we always mean a simple circuit. In contrast, our paths may include 
a vertex and or an edge more than once. 

Recall that a family $\{A_1, \ldots, A_N\}$ of complex matrices is simultaneously triangularizable if there exists an 
invertible matrix $S$ such that $S^{-1}A_iS$ are upper triangular matrices for all $1 \leq i \leq N$. A celebrated result by Frobenius asserts that any family of commuting matrices is simultaneously triangularizable \cite{Newman}. Another 
well known result of Radjavi says that a family $\{A_1, \ldots, A_N\}$ of complex matrices is simultaneously triangularizable 
if and only if $\text{trace} (A_iA_jA_l) = \text{trace} (A_lA_jA_i)$ for every $1 \leq i, j, l \leq N$. There is abundant literature 
on this topic as can be evidenced in the MR database.  A classic on this topic is the monograph \cite{Radjavi-Rosenthal} by 
Radjavi and Rosenthal. 

We point out the following papers \cite{Dubi, Shemesh, Szep, Yahaghi} in this work as we will make use of and extend  
similar results as proved in those papers to max-algebras. A brief description of (some of) the results obtained in the above 
cited papers are as follows. In \cite{Shemesh}, Shemesh proved that if matrices $A$ and $B$ are such that 
$A[A,B] = 0 \ \& \ B[A,B] = 0$ (where $[A,B]$ is the additive commutator), then $A$ and $B$ are simultaneously 
triangularizable (Theorem $2$, \cite{Shemesh}), which is closely related to Burnside's theorem (see Chapter $1$ of \cite{Radjavi-Rosenthal}). In an earlier work, Szep proved that two projectors $A$ and $B$ are simultaneously 
triangularizable if and only if their additive commutator $[A,B]$ is nilpotent (see \cite{Szep} for details). It should 
be pointed out that a more general result was obtained by Laffey much earlier. Laffey's theorem says that over an 
algebraically closed field of characteristic zero, two $n \times n$ matrices $A$ and $B$ are simultaneously 
triangularizable, if the additive commutator $C$ has rank $2$ and the matrix $A^{i}B^{j}C^{k}$ is nilpotent for 
$0 \leq i, j \leq n-1, \ 1 \leq k \leq 2$ (see the main Theorem in \cite{Laffey}). The commutant of a family also has 
interesting relationship with simultaneous triangularization. The monograph by Radjavi and Rosenthal 
\cite{Radjavi-Rosenthal} is a good source of reference on this. We point out an interesting result due to 
Yahaghi in this context. Yahaghi's main focus was to prove that a family of triangularizable compact operators has a 
hyperinvariant subspace and presents several sufficient conditions for simultaneous triangularization of a family of compact 
operators together with its commutant. Recall that a matrix is unicellular if it is triangularizable with 
a unique chain of invariant subspaces (see for instance \cite{Radjavi-Rosenthal}). We point out an interesting result by 
Yahaghi (Corollary $11$ of \cite{Yahaghi}), where it is shown that if a family $\mathcal{F}$ is simultaneously 
triangularizable, with at least one member of the family being unicellular, then the family 
$\mathcal{F} \cup \mathcal{F}^{\prime}$ (the family obtained by adjoining to $\mathcal{F}$ the commutant) is also simultaneously triangularizable. Yet another paper that we have cited above in this work is due to Dubi \cite{Dubi}. Dubi introduced the notion of the characteristic polynomial of a family $\mathcal{F}$ of complex matrices in relation to a weaker 
notion of simultaneous triangularization, which coincides with the usual notion of simultaneous triangularization when the collection consists of only $2 \times 2$ matrices.

We present a brief account of max-algebraic spectral theory and point out how commuting matrices in max-algebras 
can be simultaneously reduced to a specific block upper triangular form, known as the Frobenius normal form. We also 
point out the connection with Kleene star of a matrix $A$. 
Given a matrix $A\in M_n(\mathbb{R}_+)$, a scalar $\lambda \in \mathbb{R}_+$ and a nonzero vector 
$x\in\mathbb{R}_+^n$ form an \emph{eigen pair} of $A$ if $Ax = \lambda x$. It is not hard to verify that the eigenspectrum 
of $A$, denoted by $Eig(A)$, which is the collection of all eigenvectors of $A$ is a max-algebraic polytope (see \cite{Maclagan-Strumfels} or \cite{Morrison-Tran}). For a cycle $C = i_1\! \to i_2\! \to \cdots \to i_k\!\to i_1$ in the 
digraph $G_A$ of $A$, its weight and cycle mean are given by 
$w(C) = a_{i_1 i_2} a_{i_2 i_3} \cdots a_{i_k i_1} = \prod_{\ell=1}^k a_{i_\ell i_{\ell+1}}$ and 
$\mu(C)=\big(w(C)\big)^{1/k}$. The maximum cycle mean of $A$ (also known as the Perron root) is then 
$\lambda(A) = \max\{\mu(C): C\text{ a cycle of }G_A\}$, with the convention that $\lambda(A) = 0$ when $G_A$ 
has no circuits. The following theorem (Theorem $4.2.4$ from \cite{Butkovic}) is well known. 

\begin{proposition}\label{principal-eigenvalue}
$\lambda(A)$ is an eigenvalue for any matrix $A \in M_n(\mathbb{R}_+)$. If $A$ is irreducible, $\lambda(A)$ is the only eigenvalue of $A$.
\end{proposition}

\begin{definition}\label{Frobenius-normal-form}
A matrix $A\in M_n(\mathbb{R}_+)$ is in \emph{Frobenius normal form} (FNF) if there exists $P\in \mathrm{GL}_n(\mathbb{R}_+)$ such that
	\[
	P^{-1}AP=
	\begin{bmatrix}
		A_{11} & A_{12} & \cdots & A_{1t}\\
		0      & A_{22} & \cdots & A_{2t}\\
		\vdots & \ddots & \ddots & \vdots\\
		0      & \cdots & 0      & A_{tt}
	\end{bmatrix},
	\]
where each diagonal block $A_{ii}$ is irreducible. The blocks correspond to the strongly connected components of $G_A$. 
When every block is $1\times 1$, this reduces to an upper triangular form.
\end{definition}

Any reducible matrix can be written as a block upper triangular matrix via a permutation matrix, as above \cite{Bapat-Raghavan}. The Frobenius normal form is a max-algebraic analogue of block triangularization in classical nonnegative matrix theory. This decomposition plays a fundamental role in the spectral theory of max-algebraic matrices, 
as the Perron roots and eigenvectors can be analyzed blockwise through the Frobenius structure. A \emph{class} of $A$ 
is a strongly connected component of $G_A$. The following proposition, which follows from Theorem $4.8$ of \cite{Katz-Schneider-Sergeev}, provides a simultaneous Frobenius reduction of commuting matrices over max-algebras.

\begin{proposition}[Theorem 4.8 (i), \cite{Katz-Schneider-Sergeev}]
Let $A_1,\dots,A_r \in M_n(\mathbb{R}_+)$ be pairwise commuting matrices such that all classes of each $A_i$ have 
distinct Perron roots. Then there exists a permutation matrix $P$ such that all matrices
	\[
	P^{-1}A_iP, \quad i=1,\dots,r
	\]
are in Frobenius normal form with the same block partition.
\end{proposition}

Let 
\[
C = A_1 \oplus \cdots \oplus A_r.
\]
By Theorem $4.8$ of \cite{Katz-Schneider-Sergeev}, all matrices $A_1,\dots,A_r$ have the same classes and the same 
reduced digraph. Let $P$ be a permutation matrix that brings $C$ to Frobenius normal form so that $P^{-1}CP$ is block upper triangular with irreducible diagonal blocks corresponding to its classes. Partition each $P^{-1}A_iP$ conformally with this Frobenius block structure. then, from the proof of Theorem $4.8$ of \cite{Katz-Schneider-Sergeev}, a Frobenius form of each $A_i$ is a refinement of the Frobenius form of $C$, and since commuting irreducible blocks remain irreducible, each matrix $P^{-1}A_iP$ is also in Frobenius normal form with the same block partition. Hence, there exists a permutation matrix $P$ 
that simultaneously reduces all $A_i$ to Frobenius normal form.

If $A \in M_n(\mathbb{R}_+)$, then the series $I \oplus A \oplus A^2 \oplus \dots $ converges to 
$I \oplus A \oplus A^2 \oplus \dots A^{n-1}$ if and only if $\lambda(A) \leq 1$ (where $I$ denotes the identity matrix 
in $M_n(\mathbb{R}_+)$; this notation will be used without ambiguity in this section). This series, denoted by $A^{\ast}$, 
is called the Kleene star of $A$ and is multiplicatively idempotent. It is well known that the Kleene star of $A$ plays a 
crucial role in the description of the max-algebraic eigenvectors (see for instance \cite{Butkovic}). If $\lambda(A) > 1$, 
then we can scale the matrix $A$ by $\frac{1}{\lambda(A)}$, so that its Kleene star series converges. Then, by 
Theorem $4.6.1$ of \cite{Butkovic}, the columns of $A^{\ast}$ (after rescaling, if required) corresponding to the critical 
nodes (nodes in the corresponding digraph, which contribute to the maximum cycle mean) are max-eigenvectors of $A$ 
and they generate all eigenvectors of $A$. The Kleene star of $A$ is not distributive with respect to $\oplus$; however, 
the Kleene star of a direct sum of matrices equals the direct sum of the Kleene star of the matrices. A matrix 
$A \in M_n(\mathbb{R}_+)$ is said to be Kleene star if $A = A^{\ast}$ (so that $A$ is idempotent). The following interesting result due to Morrison and Tran, done in the min-plus algebra is worth pointing out, although we do not make explicit use of it.

\begin{theorem}\label{Thm-Morrison-Tran}
Suppose $A, B \in M_n(\mathbb{R})$ are Kleene stars. If $A \oplus B$ is a Kleene star, then
$A$ and $B$ commute, that is, $A \otimes  B = B \otimes A$. If $A \otimes B = B \otimes A$, then 
$(A \oplus B)^{{\otimes}2}$ is the Kleene star of $A \oplus B$. In particular, only for $n = 2$ and $n = 3$,  
we have: $A \otimes B = B \otimes A$ if and only if $A \oplus B = A \otimes B = B \otimes A$.
\end{theorem}

A brief description of the results obtained follows. We begin by proving that a matrix $A$ over a max algebra is 
triangularizable if and only if the associated digraph $G_A$ has no directed cycle of length at least two (see Theorem \ref{traingle}). This result is crucial to this paper and as pointed in Example \ref{example-0}, this serves as a crucial 
difference between the classical setting and the max-algebraic setting. As a consequence of Theorem \ref{traingle}, a 
topological order of $G_A$ gives the required permutation that puts the matrix in triangular form. The main results 
come next. This is subdivided into three subsections. We begin the first subsection by proving that a pair $A, B$ of 
nonnegative matrices are simultaneously triangularizable if and only if the union of their corresponding digraphs, 
$G_{A \oplus B}$, contains no directed multi-vertex cycles (see Theorem \ref{first result}). We explore several 
consequences of this result. The second subsection concerns max-commutators and commutants in relation to simultaneous triangularization of a pair of matrices. There are three main results in this subsection.  All of these 
results illustrate the similarities between the classical and the max-algebraic setting in relation to simultaneous 
triangularization. After proving some preliminary lemmas (Lemmas \ref{commutator-1}, \ref{row-anni}, and 
\ref{lem:corner}), we prove one of the main results of this subsection (Theorem \ref{thm:simtri}), which is a generalization 
of a classical result due to Shemesh (Theorem $2$, \cite{Shemesh}). We then proceed to prove that two projectors $A$ 
and $B$ are simultaneously triangularizable if their additive max-commutator is nilpotent (Theorem \ref{projector simul. triangle}), although the converse is not true (Example \ref{counterex-projector thm}). This is again a generalization of 
Szep result mentioned in the previous paragraph. The third result concerns unicellular matrices (Definition \ref{unicellular} 
and Example \ref{unicellular-example}), which we believe has not appeared in the max-algebra context, and prove that 
if $A$ and $B$ are simultaneously triangularizable with one of them unicellular, the family obtained by adjoining to the pair 
their commutant is also simultaneously triangularizable (Theorem \ref{commutant-simultaneous triangularization}). Once 
again, this is similar to a result due to Yahaghi \cite{Yahaghi}, as mentioned in the previous paragraph. In the last subsection 
of this paper, we introduce the notion of the characteristic polynomial of a pair of matrices $\{A,B\}$, denoted by 
$P_{A,B}(z)$, using the tropical determinant and prove that if $A$ and $B$ are simultaneously triangularizable, then $P_{A,B}(z)$ is a product of $n$-linear factors (Theorem \ref{sim-triangular = linear factors}). A possible motivation to 
study this from max-algebraic spectral theory is pointed out, along with an example. The converse being not true 
(Example \ref{example-2.1}), we set out to determine necessary and sufficient conditions for $P_{A,B}(z)$ to be a product 
of linear factors (Theorem \ref{nasc-linear factors}). This section ends with a corollary involving a pair of row diagonally dominant matrices, where their characteristic polynomial is a product of linear factors 
(Corollary \ref{diagonal dominance-cor}). Three algorithms, the first one for triangularizability, the second for simultaneous triangularizability and the third one for linear factorization are presented, along with basic complexity issues. Wherever 
possible, examples are provided to illustrate the results obtained. The paper ends with a section on concluding remarks, where 
we summarize the results obtained in this paper and also point out an interesting question that could be investigated in future.

\section{Preliminaries}

Some preliminary definitions and results are presented in this section. 

\begin{definition}
For a matrix $A = (a_{ij}) \in M_n(\mathbb{R}_+)$, the directed graph associated with $A$, 
denoted by $G_A$, has as vertices $\{1,\dots,n\}$ and edges $i \to j$ when $a_{ij} > 0$.
\end{definition}

\begin{definition}\cite{Deograph}
Let $G = (V,E)$ be a finite directed graph, with $|V| = n$ and edge set $E$. A \textit{topological ordering} 
$\preceq$ of $G$ is a bijection $\pi:V \rightarrow \{1,\dots, n\}$ (an ordering of the vertices $v_1,\dots,v_n$) 
such that for every edge $v_i \to v_j$ in $E$ we have $\pi(v_i) < \pi(v_j)$ (each pair of edge points  
moves forward in the ordering).
\end{definition}

The following proposition (Theorem $14-4$ of \cite{Deograph}) will be used in the sequel.

\begin{proposition}\label{topo. ordering = acyclic}
A finite directed graph admits a topological ordering if and only if it is acyclic.
\end{proposition}

We are now in a position to define and prove our first result of this paper.

\begin{definition}\label{triangularizable}
A matrix $A \in M_n(\mathbb{R}_+)$ is said to be \textit{triangularizable} if there exists a 
$P \in GL_n(\mathbb{R}_+)$ such that $P^{-1}AP$ is an upper triangular matrix. 
\end{definition}

Our first nontrivial result is the following.

\begin{theorem}\label{traingle}
A matrix $A \in M_n(\mathbb{R}_+)$ is triangularizable if and only if the corresponding digraph $G_A$ has no directed multi-vertex cycles. 
\end{theorem}

\begin{proof}
Let us label the vertices of the digraph $G_A$ as $\{1, 2, \dots, n\}$. Suppose $G_A$ is acyclic, except 
possibly for self-loops (corresponding to diagonal entries). Then, there exists a topological ordering 
$\pi: \{1,2, \cdots, n\} \rightarrow \{1,2, \cdots, n\}$, such that whenever $i \to j \in G_A$, 
$\pi(i) < \pi(j)$ for distinct $i, j$. Let $P$ be the permutation matrix corresponding to $\pi$. 
Then $\hat{A}:= P^{-1}AP$ is an upper triangular matrix. This is because when $a_{ij} >0$, there exists 
an edge $i \to j \in G_A$ and by assumption, $\pi(i) \leq \pi(j)$. Thus, in $\hat{A}$, the corresponding 
entry $\hat{a}_{\pi(i) \pi(j)} > 0$ is either on the diagonal or above the diagonal. \\

Conversely, if $A$ is triangularizable, then, there exists a permutation matrix 
$P \in GL_n(\mathbb{R}_+)$ such that $P^{-1}AP$ is upper triangular. Proposition \ref{topo. ordering = acyclic} implies 
that the  digraph $G_{(P^{-1}AP)}$ is acyclic, except possibly for self-loops. Since the digraphs $G_{(P^{-1}AP)}$ and 
$G_A$ are isomorphic, it follows that $G_A$ is also acyclic, with possible self loops.   
\end{proof}

The following example is quite pertinent, before we proceed further.

\begin{example}\label{example-0}
Let $A = \begin{pmatrix}
	              2 & 5 & 3\\
	              6 & 1 & 0\\
	              0 & 2 & 5
\end{pmatrix}$. This matrix is not triangularizable through any generalized permutation matrix $P$ as there is a 
nontrivial cycle.
\end{example}

Recall that any complex square matrix is unitarily similar to an upper triangular matrix (Theorem $2.3.1$ of 
\cite{Horn-Johnson}), whereas any real matrix is similar to a quasi-upper triangular matrix (Theorem $2.3.4$ of 
\cite{Horn-Johnson}). Example \ref{example-0} illustrates this important difference between the classical setting and 
the max algebraic setting.

The following definition is routine.

\begin{definition}\label{nilpotent}
A matrix $A \in M_n(\mathbb{R}_+)$ is \textit{nilpotent}, if $A^k = 0$ for some 
$k \in \mathbb{N}$.
\end{definition}

Nilpotent matrices are an important subset of triangularizable matrices.  From Theorem 
\ref{traingle}, we observe that nilpotent matrices also have acyclic digraphs. 

\medskip
\noindent
{\bf Remark:} From now on, we work with a collection $\{A, B\}$ of matrices, although all of the results presented 
carry over for an arbitrary collection $\mathcal{F}$ of matrices with suitable modifications.
 
\section{Main Results}

The main results are proved in this section. This is subdivided into three subsections for ease of 
reading. 

\subsection{Some preliminary results on simultaneous triangularization}\hspace*{0.5cm}

We begin with the definition of simultaneous triangularization. 

\begin{definition}\label{simultaneous triangularization}
A pair of matrices $A,B \in M_n(\mathbb{R}_+)$ are said to be \textit{simultaneously triangularizable} if there exists a $P \in GL_n(\mathbb{R}_+)$ such that both $P^{-1}AP$ and $P^{-1}BP$ are upper triangular matrices. 
\end{definition}

\begin{observation}\label{useful observation}
Let $G_A$ and $G_B$ be the digraphs of two matrices $A$ and $B$. Then, the following identity holds: 
$G_A \cup G_B = G_{A \oplus B}$. Moreover, the weight of the union of the digraphs takes the maximal weight whenever 
there is an edge between two vertices in both digraphs.
\end{observation}

The following theorem comes out as a consequence of Theorem \ref{traingle}. We nevertheless prove the 
theorem for the sake of completeness.

\begin{theorem}\label{first result}
Let $A,B \in M_n(\mathbb{R}_+)$ be two triangularizable matrices. Then, $A$ and $B$ are simultaneously triangularizable 
if and only if the union of their digraphs $G_A \cup G_B$ has no directed multi-vertex cycles.  
\end{theorem}

\begin{proof}
Let the union of the digraphs $G_A \cup G_B$ have no directed multi-vertex cycles. Recall from Observation 
\ref{useful observation}, $G_A \cup G_B = G_{A \oplus B}$. 
By Theorem \ref{traingle}, $A \oplus B$ is triangularizable. Therefore there exists a 
$P \in GL_n(\mathbb{R}_+)$ such that $P^{-1} (A \oplus B)P$ is upper triangular. Thus, 
$P^{-1}AP \oplus P^{-1}BP$ is upper triangular. Then, $[P^{-1}AP \oplus P^{-1}BP]_{ij} = 0$, for every $i>j$. Since the 
entries are nonnegative, the maximum being zero implies that both the summands are zero. Therefore, $[P^{-1}AP]_{ij} = 0$ 
and $[P^{-1}BP]_{ij} = 0$ for every $i > j$. That is, $P^{-1}AP$ and $P^{-1}BP$ are both  upper triangular matrices. 
Consequently, $A$ and $B$ are simultaneously triangularizable. 

Conversely, if $A$ and $B$ are simultaneously triangularizable, then for some invertible 
$P \in GL_n(\mathbb{R}_+)$, both $P^{-1}AP$ and $P^{-1}BP$ are upper triangular. This 
means $P^{-1} (A \oplus B)P$ is also upper triangular and hence, the digraph $G_{A \oplus B}$ 
has no directed multi-vertex cycles.  
\end{proof}

The following corollary is immediate.

\begin{corollary}\label{cor-first result}
Let $A, B \in M_n(\mathbb{R}_+)$ be two triangularizable matrices. If $G_A \subseteq G_B$ or 
$G_B \subseteq G_A$, then $A$ and $B$ are simultaneously triangularizable. 
\end{corollary}

\subsection{Max commutators and commutants}\hspace*{0.5cm}

We now investigate the connection between simultaneous triangularization and commutators and commutants of a 
pair of matrices. We begin this section with the following definition.

\begin{definition}\label{max-commutator}
Let $A,B \in M_n(\mathbb{R}_+)$. The \textit{max commutator} of $A$ and $B$ is defined by 
$C = [A,B]_{\oplus}:= AB \oplus BA$.
\end{definition}

We prove below a few lemmas in order which will be used in the main result of this section.

\begin{lemma}\label{commutator-1}
Let $A, B\in M_n(\mathbb{R}_+)$ be simultaneously triangularizable. Then, the family $\{A,B, [A,B]_{\oplus}\}$ is also simultaneously triangularizable.
\end{lemma}

\begin{proof}
Let $A$ and $B$ be simultaneously triangularizable through a generalized permutation matrix $P$ so that $P^{-1}AP$ and 
$P^{-1}BP$ are upper triangular matrices. We then have the following: 
$P^{-1}[A,B]_{\oplus}P = (P^{-1}AP) (P^{-1}BP) \oplus (P^{-1}BP) (P^{-1}AP)$. 
Since the sum and product of upper triangular matrices are again upper triangular, we see that $P^{-1}[A,B]_{\oplus}P$ is 
upper triangular. 
\end{proof}

Our next result concerns an \textit{annihilator} result when both $AC$ and $BC$ are zero.

\begin{lemma}\label{row-anni}
If $AC = 0$ and $BC = 0$, then for every index $t$, the following statements hold:
\begin{itemize}
\item[(a)] If some $a_{it}>0$ (that is, the $t^{th}$ column of $A$ is nonzero), then the $t^{th}$ row of $C$ is 
identically $0$.
\item[(b)] If some $b_{it}>0$ (that is, the $t^{th}$ column of $B$ is nonzero), then the $t^{th}$ row of $C$ is 
identically $0$.
\end{itemize}
\end{lemma}

\begin{proof}
From $AC=0$ we have $0 = (AC)_{ij} = \displaystyle \max_{k} a_{ik}c_{kj} \geq a_{it}c_{tj}$ for all $i, j$. Therefore, 
if some $a_{it}>0$, then $c_{tj}=0$ for all $j$. The second statement is proved in a similar way.
\end{proof}

\begin{lemma}\label{lem:corner}
Let $A$ and $B$ be nilpotent. If the digraph $G_{A\oplus B}$ contains a directed cycle, then there exist indices 
$u, w$ such that $C_{uw} > 0$.
\end{lemma}

\begin{proof}
Since $A$ and $B$ are nilpotent, the digraphs $G_A$ and $G_B$ contain no directed
cycles of length at least two. Suppose that $G_{A\oplus B}=G_A\cup G_B$
contains a directed cycle of length $\ell\ge 2$, say
\[
v_0 \to v_1 \to \cdots \to v_{\ell-1} \to v_0.
\]
If all the edges of this cycle belong to $G_A$, then the same vertices form a directed
cycle in $G_A$, contradicting the acyclicity of $G_A$. Similarly, the cycle cannot lie
entirely in $G_B$. Hence there is at least one edge of the cycle belonging to $G_A$
and at least one belonging to $G_B$.

Label each edge $v_i\to v_{i+1}$ (indices taken modulo $\ell$) by $A$ if it lies in
$G_A$ and by $B$ if it lies in $G_B$. This produces a cyclic word
$e_1,e_2,\cdots, e_{\ell-1}$ in the alphabet $\{A,B\}$
containing both letters. If $e_i=e_{i+1}$ for every $i$, then all $e_i$ would be identical, contradicting the fact that both $A$ and $B$ occur. Therefore, there exists an index $r$ such that
$e_r\neq e_{r+1}$. Without loss of generality, assume that
$v_r\to v_{r+1}$ is an edge of $G_A$ and $v_{r+1}\to v_{r+2}$ is an edge of $G_B$.
Set $u:=v_r$ and $w:=v_{r+2}$.

By definition of max-matrix multiplication,
\[
(AB)_{uw}=\max_k \{a_{uk}b_{kw}\}\ge a_{u\,v_{r+1}}\,b_{v_{r+1}\,w}.
\]
Since both $a_{u\,v_{r+1}}$ and $b_{v_{r+1}\,w}$ are strictly positive, it follows
that $(AB)_{uw}>0$. Hence
\[
C_{uw}=\max\{(AB)_{uw},(BA)_{uw}\}\ge (AB)_{uw}>0,
\]
and therefore the $u$-th row of $C$ is nonzero.	
\end{proof}

Combining the above lemmas, we have the following theorem.

\begin{theorem}\label{thm:simtri}
Let $A, B\in M_n(\mathbb{R}_+)$ be nilpotent. If $AC = BC = 0$, then $A$ and $B$ are simultaneously 
triangularizable.
\end{theorem}

\begin{proof}
It suffices to prove that the digraph $G_{A\oplus B}$ is acyclic. Suppose $G_{A\oplus B}$ has a directed cycle. By 
Lemma \ref{lem:corner}, there is a vertex $u$ on the cycle such that row $u$ of $C$ is nonzero. But on this cycle, 
$u$ also has an incoming edge and therefore, column $u$ is nonzero in $A$ or in $B$. Lemma \ref{row-anni} then forces 
row $u$ of $C$ to be zero, a contradiction. Hence $G_{A\oplus B}$ is acyclic. Finally, a topological ordering of 
$G_{A\oplus B}$ triangularizes both $A$ and $B$.
\end{proof}

The following corollary is immediate.

\begin{corollary}
Let $A,B\in M_n(\mathbb{R}_+)$ be nilpotent such that both $A$ and $B$ annihilate $AB$ and $BA$. Then 
$A$ and $B$ are simultaneously triangularizable.
\end{corollary}

We now bring out simultaneous triangularization of projector matrices.

\begin{definition}\label{projector}
A matrix $A \in M_n(\mathbb{R}_+)$ is a \textit{projector or idempotent}, if $A^2 = A$.
\end{definition}

We prove below that for two triangularizable projector matrices $A$ and $B$, nilpotency of the max-commutator implies 
that $A$ and $B$ are simultaneously triangularizable. However, the converse fails.

\begin{theorem}\label{projector simul. triangle}
Let $A, B \in M_n(\mathbb{R}_{+})$ be triangularizable projector matrices. If the max-commutator $C=[A,B]_{\oplus}$ is nilpotent, then $A$ and $B$ are simultaneously triangularizable.
\end{theorem}

\begin{proof}
The proof relies on the crucial fact that the digraphs of projector matrices have a transitivity property: 
if there exists an edge from $i \to k$ and $k \to j$, then there exists an edge from $i \to j$ in $G_A$, as 
$a_{ij} \geq a_{ik}a_{kj}$.  Moreover, if there exists an edge $i \to k$ in $G_A$ and $k \to j$ in $G_B$, then there exists 
an edge $i \to j$ in $G_{AB}$.
	
Let us now assume that the max-commutator $C$ is nilpotent so that $G_C$ is acyclic. To prove $A$ and $B$ are 
simultaneously triangularizable, it suffices to prove that $G_{A \oplus B}$ is acyclic. Assume on the contrary that 
$G_{A \oplus B}$ has a cycle of length atleast $2$, say, $v_1 \to v_2 \to \cdots \to v_m \to v_1$. 
Since $A$ and $B$ are individually triangularizable, $G_A$ and $G_B$ have no directed cycles of length $\geq 2$.  
Therefore, in this cycle, there must be atleast one edge from $G_A$ and atleast one from $G_B$. Now, partition the cycle 
into maximal runs of consecutive edges belonging to the same graph, say, 
\begin{align*}
&v_1 \xrightarrow{A} v_2 \xrightarrow{A} \cdots \xrightarrow{A} v_{k_1}\\
&v_{k_1} \xrightarrow{B} v_{k_1 +1} \xrightarrow{B} \cdots \xrightarrow{B}v_{k_2}\\
&\vdots\\
&v_{k_l} \xrightarrow{A / B} v_{k_l +1} \xrightarrow{A/ B} \cdots \xrightarrow{A / B} v_{k_{l+1}} = v_1.
\end{align*}
	
Then, from transitive property of $G_A$ and $G_B$ in the above blocks, we get 
$v_1 \xrightarrow{A} v_{k_1}$, $v_{k_1} \xrightarrow{B} v_{k_2}$ and so on. Thus, we get an alternating cycle in 
$G_{A \oplus B}$ in which each of the consecutive edges comes from $G_A$ or $G_B$ alternatively. Let us assume without 
loss of generality that the cycle is the following: 
\[v_1 \xrightarrow{A} v_{k_1} \xrightarrow{B} v_{k_2} \xrightarrow{A}  v_{k_3} \xrightarrow{B} \cdots \xrightarrow{A} v_1.\] 
As explained in the very first step that when $G_A$ has an edge $i \to k$ and $G_B$ has an edge $k \to j$,
then there is an edge $i \to j$ in $G_{AB}$; from this we get a cycle in $C = [A, B] _{\oplus}$ 
as follows: 
\[v_1 \xrightarrow{C} v_{k_2} \xrightarrow{C} v_{k_4}  \xrightarrow{C} \cdots \xrightarrow{C} v_1.\] 
This contradicts the initial assumption that $C$ is nilpotent. Thus, $G_{A \oplus B}$ is acyclic, thereby proving that 
$A$ and $B$ are simultaneously triangularizable.
\end{proof}

The following example illustrates that the converse of the above theorem fails to hold.

\begin{example}\label{counterex-projector thm}
Let $A = \begin{pmatrix}
	             1 & 1 & 0\\
	             0 & 1 & 0\\
	             0 & 0 & 0
               \end{pmatrix}, \  \text{and} \ B = \begin{pmatrix}
               																		1 & 0 & 1\\
               																		0 & 0 & 0\\
               																		0 & 0 & 1
                                                                                   \end{pmatrix}$. $A$ and $B$ are projectors, triangular and also 
simultaneously triangularizable as $G_{A \oplus B} = G_A \cup G_B$ is acyclic. However, the max-commutator 
$[A,B]_{\oplus} = \begin{pmatrix}
	                                1 & 1 & 1\\
	                                0 & 0 & 0\\
	                                0 & 0 & 0
                                  \end{pmatrix}$, which is not nilpotent.
\end{example}

Let us now see how the commutant of $A$ and $B$ relates to simultaneous triangularization of these matrices. Let 
$V = \{v_1, \dots, v_k\} \subset \mathbb{R}^n_+$. The max-linear cone spanned by the vectors in $V$ is defined as 
$$\langle v_1, \dots, v_k \rangle:= \displaystyle \Biggl\{\bigoplus_{i=1}^{k} \lambda_i v_i : \lambda_i \geq 0 \Biggr\}.$$ 
We shall denote the max-linear cone spanned by $k$ vectors by $U_k$. 

\begin{definition}\label{unicellular}
A matrix $A \in M_n(\mathbb{R}_+)$ is called \textit{unicellular} if it has a unique triangularizing chain. 
\end{definition}

Therefore, there is a unique permutation $\pi$ of $\{1, 2, \dots, n\}$ such that the max-linear spanning cone of the 
vectors $e_{\pi(i)}, \ U_i = \langle e_{\pi(1)}, e_{\pi(2)}, \cdots, e_{\pi(k)}\rangle$ forms a unique chain of invariant cones 
for $A$: $\{0\} \subset U_1 \subset U_2 \subset \cdots \subset U_n =\mathbb{R}_+^{n}$. This is equivalent to saying that 
the corresponding digraph $G_A$ has a unique topological ordering of its vertices $1, 2, \cdots, n$. The following 
example illustrates this.

\begin{example}\label{unicellular-example}
Let $A =\begin{pmatrix}
        1&0&7&0\\
        0&1&0&5\\
        0&0&1&0\\
        6&0&0&1
    \end{pmatrix}$. Then, $G_A$ has only one directed path (except self-loops), which is given by 
$2 \rightarrow 4 \rightarrow 1 \rightarrow 3$. Since there are no other paths and branchings, there is a unique 
topological order for the vertices given by $2 \preceq 4 \preceq 1\preceq 3$. Thus, the support of the unique  
triangularizing chain is the following:
$$\emptyset \subset \{2\} \subset \{2, 4\} \subset \{2, 4, 1\} \subset \{2, 4,1,3\} .$$
We can then take the permutation $\pi = (1, 2, 4,3)$ and the corresponding permutation matrix $P$ and compute 
$P^{-1}AP =  \begin{pmatrix}
        1&5&0&0\\
        0&1&6&0\\
        0&0&1&7\\
        0&0&0&1
\end{pmatrix}$ to be upper triangular.
\end{example}

\begin{definition}\label{commutant}
The commutant of a pair of matrices $\{A,B\}$, denoted by, is the set  $\{A,B\}': = \{X:  AX = XA,\ BX = XB\}$.
\end{definition}

Notice that the above definition is very similar to the classical one, except that matrix multiplication is in the max sense. 
The final theorem of this section is the following.

\begin{theorem}\label{commutant-simultaneous triangularization}
Let $A, B \in M_n(\mathbb{R}_+)$ be simultaneously triangularizable. If at least one of them (say $A$) is unicellular. 
Then the family $\{A,B\}  \cup \{A,B\}'$ is simultaneously triangularizable. 
\end{theorem}

\begin{proof}
Since $A$ is unicellular, the digraph $G_A$ induces a total order on $\{1,\dots, n\}$ and hence $G_A$ admits a unique 
topological ordering $\pi$. Let $P$ be the permutation matrix corresponding to the permutation $\pi$. Then 
$\bar{A}:= P^{-1}AP$ is upper triangular; since $A$ and $B$ are simultaneously triangularizable, $\bar{B}:= P^{-1}BP$ 
is also upper triangular. For $k = 0, 1,\dots, n$, define $U_k = \langle e_{\pi(1)},\dots, e_{\pi(k)}\rangle$.
The $U_k$'s form the unique sequence of $A$--invariant max-cones:
$\{0\}=U_0 \subset U_1 \subset \cdots \subset U_n = \mathbb{R}_+^{n}$. Now let $X \in \{A,B\}'$, so that 
$XA=AX$ and $XB=BX$. From $XA=AX$ we infer that $X(U_k)$ is $A$--invariant for each $k$. By unicellularity of $A$, 
the only $A$--invariant subcones are the $U_k$ and hence $X(U_k)\subseteq U_k$ for all $k$. Thus, $U_k$ is an 
invariant max-cone of $X$ for each $k = 1, \dots, n$. We then have $P^{-1}XPe_k = P^{-1}Xe_{\pi(k)}$. Since 
$X(U_k) \subseteq U_k$, the vector $Xe_{\pi(k)}$ is supported on $\{\pi(1), \dots, \pi(k)\}$. Applying $P^{-1}$, we see 
that the support of $P^{-1}XPe_k$ is in the first coordinates. This makes all the entries of $P^{-1}XP$ below the diagonal 
to be zero, thereby making it upper triangular. Hence $\{A,B\}\cup\{A,B\}'$ is simultaneously triangularizable.
\end{proof}

\subsection{The characteristic polynomial of a pair and its linear factorization}\hspace*{0.5cm}

In this section, we define the notion of characteristic polynomial of a pair of 
matrices in terms of the tropical determinant and investigate when it can be factored into linear terms. 
This has an interesting connection with simultaneous triangularization. Notice that in the one variable case, the roots 
of the characteristic polynomial of a max-algebraic matrix $A$, are given by the intersection points of the curves (lines in 
case of max-plus algebra) corresponding to each monomial term in the polynomial $\text{tdet}(I \oplus zA)$. 
In that case, the largest root is the principal eigenvalue of $A$. This takes us to the following definitions when there are 
atleast two matrices. We shall make use of the following definitions, stated in order of preference. 
The first one is that of the tropical determinant.

\begin{definition}\label{t-det}
Let $A=(a_{ij})\in M_n(\mathbb{R}_+)$. The \textit{tropical determinant} of $A$ is defined to be 
$\text{tdet}(A) = \displaystyle \bigoplus_{\sigma\in S_n} \displaystyle \prod_{i=1}^{n} a_{i \sigma(i)}
= \max_{\sigma\in S_n}\ \displaystyle \prod_{i=1}^{n} a_{i \sigma(i)}$.
\end{definition}

We now define the characteristic polynomial of a pair of max-algebraic matrices.

\begin{definition}\label{char poly of a pair}
Let $A, B\in M_n(\mathbb{R}_+)$ and for $z = (z_1,z_2) \in \mathbb{R}^2$, set
$M(z):= I \ \oplus\ z_1A\ \oplus\ z_2B, \ z_1,z_2 \geq 0$. The \textit{characteristic polynomial} of $A$ and $B$ is 
defined to be: $P_{A,B}(z):=\text{tdet}\!\left(M(z)\right)$.
\end{definition}

The following two definitions are that of an eigenvalue pair of a pair of max-algebraic matrices and the associated 
eigenvector.

\begin{definition}\label{def:eigenpair-sing}
Let $A, B \in M_n(\mathbb{R}_+), M(z) = I\oplus z_1A\oplus z_2B$ and $P_{A,B}(z) = \text{tdet}(M(z))$.
A point $z = (z_1,z_2) \in \mathbb{R}_+^2$ is called an eigenvalue pair of $(A,B)$
if $M(z)$ is tropically singular; that is, if the maximum in $\text{tdet}(M(z))$ is attained by
at least two distinct permutations.
\end{definition}

\begin{definition}\label{def:eigenvector-kernel}
For an eigenvalue pair $z$, a vector $0 \neq v$ is called an eigenvector associated
with $z$ if for each row $i$, the maximum in
$(M(z)\otimes v)_i =  \displaystyle \max_{1 \leq j \leq n} \{m_{ij}(z) \,v_j\}$ is attained atleast twice.
\end{definition}

We illustrate the above definition by means of a nice application in contingency platforming under two operational support packages. The set up is the following: A station must assign three trains $T_1,T_2,T_3$ to three platforms $P_1,P_2,P_3$ (one-to-one) during a disruption window and the railway department has a punctuality target of departure within two 
minutes of the schedule. Let us assume that $p_i$ is the historical probability that $T_i$ meets this target when it stays on its scheduled platform. To compare this contingency reassignments on a common scale, it is natural to use the on-time probability ratio 
\[
r_{ij}=\frac{\Pr(T_i \text{ on-time on } P_j)}{\Pr(T_i \text{ on-time on its scheduled platform})}.
\] 
This means that the default assignment $T_i\mapsto P_i$ is normalized to $r_{ii}=1$, where $r_{ij}>1$
means that reassigning $T_i$ to $P_j$ improves on-time performance relative to the default. At the station, two operational support packages are available:

\begin{itemize}
\item \textbf{Support A (passenger guidance).}
$z_1 \in [0,1]$ is the fraction of the planned passenger-guidance package that is actually deployed:
$z_1=0$ means no extra staff or signs or announcements are used, and $z_1=1$ means that the fully planned package is used.
	
\item \textbf{Support B (dispatch priority).}
$z_2 \in [0,1]$ is the fraction of the planned dispatch-priority package that is actually available:
$z_2=0$ means that no priority can be granted, and $z_2=1$ means that fully planned priority capacity is available
to give selected trains precedence at route-setting or signal conflicts.
\end{itemize}

At full deployment ($z_1=z_2=1$), data for the present disruption indicates the following:

\[
p_1=0.60,\quad p_2=0.65,\quad p_3=0.55,
\]
and the following feasible re-routings with their measured ratios:
\[
T_1\to P_2:\ \frac{0.78}{0.60}=1.30,\qquad
T_2\to P_3:\ \frac{0.78}{0.65}\approx 1.20,\qquad
T_3\to P_1:\ \frac{0.77}{0.55}=1.40.
\] 

All other non-scheduled reassignments are deemed infeasible (or unacceptable) in this window and are assigned the 
ratio $0$. For each reassignment $T_i\to P_j$ ($i\neq j$), we decide in advance which support package is responsible for 
enabling that reassignment in the disruption window. If the reassignment is enabled by passenger-guidance actions 
(Support A), then its full-deployment on-time ratio is placed in $A$;  that is, $a_{ij} = r_{ij}$ and $b_{ij} = 0$. A 
similar definition works for $B$. All reassignments that are not feasible or allowed in the window are assigned ratio $0$ 
in both matrices.

In this disruption window, the reroutings $T_1\to P_2$ and $T_2\to P_3$ primarily requires managing
passenger movement across platforms and hence are enabled by passenger guidance and are recorded in $A$,
whereas the rerouting $T_3\to P_1$ primarily requires resolving an approach or route conflict by traffic control
 and hence is enabled by dispatch priority and is recorded in $B$.

We encode these full-deployment ratios by
\[
A=\begin{pmatrix}
	0 & 1.30 & 0\\
	0 & 0 & 1.20\\
	0 & 0 & 0
\end{pmatrix},
\qquad
B=\begin{pmatrix}
	0 & 0 & 0\\
	0 & 0 & 0\\
	1.40 & 0 & 0
\end{pmatrix}.
\]
For deployment levels $(z_1,z_2)$ the effective max-times assignment matrix is given by the pencil
\[
M(z)=I\oplus z_1A\oplus z_2B=
\begin{pmatrix}
	1 & 1.30z_1 & 0\\
	0 & 1 & 1.20z_1\\
	1.40z_2 & 0 & 1
\end{pmatrix}.
\]

For given deployment levels $(z_1,z_2)$, the entry $m_{ij}(z)$ is the best available on-time ratio for assigning 
$T_i$ to $P_j$ under the default plan (which is given by the identity matrix $I$) and the two support packages 
($z_1A$ and $z_2B$). The tropical determinant
\[
P_{A,B}(z_1,z_2)=\text{tdet}(M(z))
\]
is then the best station-wide ratio that can be achieved by a one-to-one
platform assignment. In this example, only two complete assignment plans can occur:
the default plan (identity permutation) with value $1$, and the rotation
$T_1\mapsto P_2,\ T_2\mapsto P_3,\ T_3\mapsto P_1$ with value
$(1.30z_1)(1.20z_1)(1.40z_2)=2.184\,z_1^2z_2$. Hence
\[
P_{A,B}(z_1,z_2)=\max\{\,1,\ 2.184\,z_1^2z_2\,\}.
\]

This gives us the operational decision immediately:

\[
\begin{cases}
	2.184\,z_1^2z_2<1 \quad\Rightarrow\quad \text{keep the scheduled platforming (default plan),}\\[2mm]
	2.184\,z_1^2z_2>1 \quad\Rightarrow\quad \text{use the rotated platforming plan.}
\end{cases}
\]
The curve
\[
1=2.184\,z_1^2z_2
\]

is exactly the boundary where the two plans are equally good. By Definition \ref{def:eigenpair-sing}, points on this curve 
are the eigenvalue pairs: they are the deployment levels at which $M(z)$ becomes tropically singular, because two
distinct permutations attain the same maximum in $\text{tdet}(M(z))$. Along an eigenvalue pair $z=(z_1,z_2)$, an 
associated eigenvector (Definition \ref{def:eigenvector-kernel}) gives a concrete certificate that the station is
exactly at the switching boundary: in each train-row, the best available choice is not unique.

For the matrix
\[
M(z)=
\begin{pmatrix}
	1 & 1.30z_1 & 0\\
	0 & 1 & 1.20z_1\\
	1.40z_2 & 0 & 1
\end{pmatrix},
\]
we say, $v= (v_1,v_2,v_3)^T$ is an eigenvector, if the \emph{max attained at least twice} condition in each row:
\[
1\cdot v_1=(1.30z_1)\,v_2,\qquad
1\cdot v_2=(1.20z_1)\,v_3,\qquad
(1.40z_2)\,v_1=1\cdot v_3.
\]
These three equalities hold simultaneously if and only if
\[
1=2.184\,z_1^2z_2,
\]
which is precisely on the eigenvalue-pair curve. For instance, choosing $v_3=1$ gives the explicit associated eigenvector
\[
v=
\begin{pmatrix}
	1.56\,z_1^2\\[1mm]
	1.20\,z_1\\[1mm]
	1
\end{pmatrix},
\qquad \text{valid whenever } \ 1=2.184\,z_1^2z_2.
\]
Operationally, the three tie equalities mean:
\begin{itemize}
\item for $T_1$, stay on $P_1$ and move to $P_2$ using Support A are equally good;
\item for $T_2$, stay on $P_2$ and move to $P_3$ using Support A are equally good;
\item for $T_3$, stay on $P_3$ and move to $P_1$ using Support B are equally good.
\end{itemize} 

Hence the eigenvector identifies the competing actions that are simultaneously optimal at the
boundary; moving slightly off the boundary makes one of the two assignment plans better.

The root set of $P_{A,B}(z)$ consists exactly of those parameter values $z=(z_1,z_2) \in \mathbb{R}_+^2$, 
for which the maximum in $\text{tdet}(M(z))$ is attained by atleast two distinct permutations; equivalently, $M(z)$ is 
tropically singular.  This is the natural two-parameter analogue of the one-variable characteristic max-polynomial 
of Cuninghame--Green, whose roots are often viewed as \emph{algebraic eigenvalues}. These have been studied in detail in  \cite{CuninghameGreen,Nishida-Watanabe}.  In this sense, Definitions \ref{def:eigenpair-sing} and \ref{def:eigenvector-kernel} describe \emph{algebraic singular pairs} $(z_1,z_2)$ and their associated tropical-kernel. This notion should not be confused with the \emph{two-sided generalized eigenproblem} $A\otimes x=\lambda\otimes B\otimes x$, whose spectrum 
$\sigma(A,B)$ may be a \emph{finite union of intervals} (and not merely a finite set) \cite{Gaubert-Sergeev}, and in fact 
any finite family of intervals (and isolated points) can occur as $\sigma(A,B)$ for a suitable pair $(A,B)$ \cite{Sergeev}. 
Thus, our eigenvalue pairs are determinantal or algebraic singularity parameters of the pencil
$M(z)$, whereas $\sigma(A,B)$ captures solvability of a two-sided eigenproblem and generally exhibits interval-valued 
behavior. In the two variable case, we observe that the roots of the characteristic polynomials are the intersection curves of the 
surfaces corresponding to the monomials in the polynomial $\text{tdet}(I \oplus z_1A \oplus z_2B)$. 
 
Our first result connecting simultaneous triangularization and linear factorization of the corresponding characteristic 
polynomial is the following.

\begin{theorem}\label{sim-triangular = linear factors}
Let $A, B \in M_n(\mathbb{R}_+)$ be simultaneously triangularizable. Then, their characteristic polynomial 
$P_{A,B}(z)$ is a product of $n-$linear factors. 
\end{theorem}

\begin{proof}
If $A,B$ are simultaneously triangularizable, there is a permutation matrix 
$P \in GL_n(\mathbb{R}_+)$, such that $A^{\prime} = P^{-1}AP,\ \& \  B^{\prime} = P^{-1}BP$ 
are upper triangular. Observe that $M'(z) = I \oplus z_1 A' \oplus z_2 B'$ is also upper triangular. Then 
the tropical determinant of $M^{\prime} (z)$ given by $\text{tdet}(M'(z)) = 
\displaystyle \prod_{i=1}^{n} \bigl(1\oplus a'_{ii}z_1\oplus b'_{ii}z_2\bigr)$ (note that for any 
other permutation $\sigma \neq id$, atleast one of $a'_{i\sigma(i)}\oplus b'_{i\sigma(i)} = 0$, 
thereby making the entire product zero. 
\end{proof}

The converse of the above theorem is not true in the max-algebraic setting, as the following example 
illustrates. 

\begin{example}\label{example-2.1}
	Let
	$A=\begin{pmatrix}
		2 & 2 & 0 & 0\\
		0 & 3 & 2 & 0\\
		0 & 0 & 4 & 2\\
		0 & 0 & 0 & 5
	\end{pmatrix}
	\quad \text{and} \quad
	B=\begin{pmatrix}
		4 & 0 & 0 & 0\\
		0 & 5 & 0 & 0\\
		0 & 0 & 6 & 0\\
		2 & 0 & 0 & 7
	\end{pmatrix}$.
	It is clear that $A$ and $B$ are not simultaneously triangularizable. Indeed, $G_A$ contains the edges
	$1\to 2$, $2\to 3$, $3\to 4$, while $G_B$ contains the edge $4\to 1$. Hence the union digraph
	$G_{A\oplus B}=G_A\cup G_B$ contains the directed cycle $1\to 2\to 3\to 4\to 1$.
	
	We then have
	\[
	M(z)=I\oplus (z_1A)\oplus (z_2B)=
	\begin{pmatrix}
		1\oplus 2z_1\oplus 4z_2 & 2z_1 & 0 & 0\\
		0 & 1\oplus 3z_1\oplus 5z_2 & 2z_1 & 0\\
		0 & 0 & 1\oplus 4z_1\oplus 6z_2 & 2z_1\\
		2z_2 & 0 & 0 & 1\oplus 5z_1\oplus 7z_2
	\end{pmatrix}.	\] 
	Since the only nonzero off-diagonal entries of $M(z)$ are $m_{12},m_{23},m_{34},m_{41}$, the tropical
	determinant is attained either by the identity permutation or by the $4$--cycle $\sigma=(1\,2\,3\,4)$. Therefore,
	\begin{align*}
		\text{tdet}(M(z)) &=
		\max\!\left(
		\prod_{i=1}^4 (1\oplus a_{ii}z_1\oplus b_{ii}z_2),\ 2z_1\cdot 2z_1 \cdot 2z_1 \cdot 2z_2
		\right)\\
		&=
		\max\!\left(
		\prod_{i=1}^4 (1\oplus a_{ii}z_1\oplus b_{ii}z_2),\ 16\,z_1^{3}z_2
		\right).
	\end{align*} But
	\[
	\prod_{i=1}^4 (1\oplus a_{ii}z_1\oplus b_{ii}z_2)
	\ \ge\
	2z_1 \cdot 3z_1\cdot 4z_1\cdot 7z_2
	=168\,z_1^{3}z_2
	\ \ge\
	16\,z_1^{3}z_2,
	\]
and hence $\text{tdet}(M(z))=\prod_{i=1}^4 (1\oplus a_{ii}z_1\oplus b_{ii}z_2)$ for all $z_1,z_2\ge 0$.
Thus $P_{A,B}(z)=\text{tdet}(M(z))$ is a product of linear factors even though $A$ and $B$ are not simultaneously triangularizable.
\end{example}

We now provide a necessary and sufficient condition for the characteristic polynomial to factor into 
linear terms.

\begin{theorem}\label{nasc-linear factors}
Let $A, B\in M_n(\mathbb{R}_+)$ be triangularizable matrices. Then the following statements are equivalent:
\begin{enumerate}
\item The tropical determinant of $A\oplus B$ comes from the identity permutation; that is, 
$\text{tdet} (A \oplus B) = \displaystyle \prod_{i=1}^{n} (a_{ii} \oplus b_{ii})$.
\item $P_{A,B}(z)$ is a product of linear factors; that is, 
$P_{A,B}(z) = \displaystyle \prod_{i=1}^{n} \left(1\ \oplus\ \alpha_i z_1\ \oplus\ \beta_i z_2\right)$ 
for some $\alpha_i, \beta_i \geq 0$.
\end{enumerate}
\end{theorem}

\begin{proof}
$(1) \Rightarrow(2)$: Assume that $\text{tdet} (A \oplus B) = 
\displaystyle \prod_{i=1}^{n} (a_{ii} \oplus b_{ii})$. We prove that the tropical determinant of 
$M(z)$ also comes from the identity permutation. Assume on the contrary that, it comes from a 
non-identity permutation $\sigma$. Let us assume without loss of generality that 
$\sigma = (i_1,i_2, \cdots, i_r)$ for some $r \geq 2$ and that $i_{r+1}, \cdots, i_n$ are fixed points of $\sigma$. Then, 

\begin{align*}    
P_{A,B}(z) & = \displaystyle \prod_{i=1}^{n} \left(I_{i\sigma(i)} \oplus z_1 a_{{i\sigma(i)}} 
\oplus z_2 b_{i\sigma(i)} \right)\\
& = \displaystyle \prod_{i=1}^{r} \left(z_1 a_{{i\sigma(i)}} \oplus z_2 b_{i\sigma(i)} \right) 
\displaystyle \prod_{i=r+1}^{n} \left(1\oplus z_1 a_{{ii}} \oplus z_2 b_{ii} \right)\\
& \geq  \displaystyle \prod_{i=1}^n \left(1 \oplus z_1 a_{{ii}} \oplus z_2 b_{ii} \right).
\end{align*}

This in turn gives, \[\prod_{i=1}^r \left(z_1 a_{{i\sigma(i)}} \oplus z_2 b_{i\sigma(i)} \right) \geq 
\displaystyle \prod_{i=1}^{r} \left(1 \oplus z_1 a_{{ii}} \oplus z_2 b_{ii} \right) \geq 
\displaystyle \prod_{i=1}^{r} \left( z_1 a_{{ii}} \oplus z_2 b_{ii} \right).\]

Now consider the following ratio for the ray $z_1= z_2 = t$. We then have, 
$$\dfrac{\displaystyle \prod_{i=1}^{r} \left(z_1 a_{{i\sigma(i)}} \oplus z_2 b_{i\sigma(i)} \right)}
{\displaystyle \prod_{i=1}^{r} \left( z_1 a_{{ii}} \oplus z_2 b_{ii} \right)} = 
\dfrac{\displaystyle \prod_{i=1}^{r} \left( a_{{i\sigma(i)}} \oplus  b_{i\sigma(i)} \right)}
{\displaystyle \prod_{i=1}^{r} \left( a_{{ii}} \oplus  b_{ii} \right)} \geq 1.$$

This gives, $\displaystyle \prod_{i=1}^{n} \left( a_{{i\sigma(i)}} \oplus  b_{i\sigma(i)} \right) 
\geq \displaystyle \prod_{i=1}^{n} \left( a_{{ii}} \oplus  b_{ii} \right)$, contradicting the initial assumption. 
Thus,  $P_{A,B}(z) = \displaystyle \prod_{i=1}^{n} \left(1\ \oplus\ a_{ii} z_1\ \oplus\ b_{ii} z_2\right)$.

$(2)\Rightarrow(1)$: Assume that
\[
P_{A,B}(z)=\prod_{i=1}^n \bigl(1\oplus a_{ii}z_1\oplus b_{ii}z_2\bigr)
\qquad \text{for all } z\in\mathbb{R}_+^2.
\]
The above holds in particular for $z_1 = z_2 = t$ for $t \in \mathbb{R}_+$.
Set $D:=A\oplus B$ and $d_{ij}:=a_{ij}\oplus b_{ij}$. For $t>0$, we have
\[
P_{A,B}(t,t) = \text{tdet}\!\bigl(I\oplus tA\oplus tB\bigr) = \text{tdet}\!\bigl(I\oplus t(A\oplus B)\bigr) 
=\text{tdet}(I\oplus tD).\] 
On the other hand, by the assumed factorization,
\[
P_{A,B}(t,t)=\prod_{i=1}^n \bigl(1\oplus t(a_{ii}\oplus b_{ii})\bigr)
=\prod_{i=1}^n \bigl(1\oplus t d_{ii}\bigr).
\] 
For a permutation $\sigma\in S_n$, define,
\[
f_\sigma(t):=\prod_{i=1}^n \bigl(I_{i\sigma(i)}\oplus t\,d_{i\sigma(i)}\bigr).
\] 
Then $\text{tdet}(I\oplus tD)=\max_{\sigma\in S_n} f_\sigma(t)$, while the identity term equals,
\[
f_{\mathrm{id}}(t)=\prod_{i=1}^n (1\oplus t d_{ii}).
\]
Since $\text{tdet}(I\oplus tD)=f_{\mathrm{id}}(t)$ for all $t>0$, we have,
\[
f_{\mathrm{id}}(t)\ \ge\ f_\sigma(t), \qquad \forall\,\sigma\in S_n,\ \forall\,t>0.
\]
This gives,
\begin{equation}\label{permutation-ratio-inequality}
\frac{f_{\mathrm{id}}(t)}{t^n}\ \ge\ \frac{f_\sigma(t)}{t^n}, \qquad \forall\,\sigma,\ \forall\, t>0.
\end{equation}

Now observe that,
\[
\frac{1\oplus t d_{ii}}{t}=\frac{\max\{1,td_{ii}\}}{t}=\max\left\{\frac{1}{t},\,d_{ii}\right\}
=\left(\frac{1}{t}\right)\oplus d_{ii},
\]
and for $i\neq \sigma(i)$, we have $I_{i\sigma(i)}=0$, so that
\[
\frac{I_{i\sigma(i)}\oplus t d_{i\sigma(i)}}{t}=\frac{t d_{i\sigma(i)}}{t}=d_{i\sigma(i)}.
\]
Hence, for every $\sigma$,
\[
\lim_{t\to\infty}\frac{f_{\mathrm{id}}(t)}{t^n}=\prod_{i=1}^n d_{ii},
\qquad
\lim_{t\to\infty}\frac{f_{\sigma}(t)}{t^n}=\prod_{i=1}^n d_{i\sigma(i)}.
\] 
Taking the limit as $t\to\infty$ in inequality \eqref{permutation-ratio-inequality} gives,
\[
\prod_{i=1}^n d_{ii}\ \ge\ \prod_{i=1}^n d_{i\sigma(i)}, \qquad \forall\,\sigma\in S_n.
\] 
Therefore,
\[
\text{tdet}(D)=\max_{\sigma\in S_n}\prod_{i=1}^n d_{i\sigma(i)}=\prod_{i=1}^n d_{ii}
=\prod_{i=1}^n (a_{ii}\oplus b_{ii}),
\]
which is exactly (1).
\end{proof}

From Example \ref{example-2.1}, we observe that the tropical determinant of $A\oplus B$ is given by the identity 
permutation. We have already seen that $P_{A,B}(z)$ is factored linearly, a direct consequence of Theorem 
\ref{nasc-linear factors}. The following is another example to illustrate this. 

\begin{example}\label{example-4}
	Let $A=\begin{pmatrix}
		10 & 0 & 2 & 0\\
		0 & 7 & 0 & 0\\
		0 & 0 & 0 & 0\\
		0 & 0 & 0 & 8
	\end{pmatrix}
	\quad \text{and} \quad
	B=\begin{pmatrix}
		4 & 0 & 0 & 0\\
		0 & 5 & 0 & 0\\
		3 & 0 & 9 & 0\\
		0 & 0 & 0 & 8
	\end{pmatrix}.$ Then $G_A$ has the single off-diagonal edge $1\to 3$, and $G_B$ has the single off-diagonal
edge $3\to 1$. In particular, $G_A$ and $G_B$ contain no directed multi-vertex cycles and 
so are triangularizable.
	
Set $D:=A\oplus B$. Then
	\[
	D=\begin{pmatrix}
		10 & 0 & 2 & 0\\
		0 & 7 & 0 & 0\\
		3 & 0 & 9 & 0\\
		0 & 0 & 0 & 8
	\end{pmatrix}.
	\] Since the only nonzero off-diagonal entries of $D$ are $d_{13}$ and $d_{31}$, the tropical
determinant $\text{tdet}(D)$ can be attained only by the identity permutation or by the transposition
	$\sigma=(1\,3)$. Therefore,
	\[
	\text{tdet}(D)=\max\!\left(d_{11}d_{22}d_{33}d_{44},\ d_{13}d_{31}d_{22}d_{44}\right)
	=\max(10\cdot 7\cdot 9\cdot 8,\ 2\cdot 3\cdot 7\cdot 8)=5040.
	\]
Hence $\text{tdet}(A\oplus B)=\prod_{i=1}^4 (a_{ii}\oplus b_{ii})$, and condition (1) of 
Theorem \ref{nasc-linear factors} holds.
	
Now, for $z=(z_1,z_2)\in\mathbb{R}_+^2$, we have
	\[
	M(z)=I\oplus z_1A \oplus z_2B =
	\begin{pmatrix}
		1\oplus 10z_1\oplus 4z_2 & 0 & 2z_1 & 0\\
		0 & 1\oplus 7z_1\oplus 5z_2 & 0 & 0\\
		3z_2 & 0 & 1\oplus 9z_2 & 0\\
		0 & 0 & 0 & 1\oplus 8z_1\oplus 8z_2
	\end{pmatrix}.
	\] 
Once again, the only permutations that can contribute are $\mathrm{id}$ and $\sigma=(1\,3)$. 
Hence,
	\[
	\text{tdet}(M(z))=
	\max\!\left(
	\prod_{i=1}^4 (1\oplus a_{ii}z_1\oplus b_{ii}z_2),\ (2z_1)(3z_2)\,(1\oplus 7z_1\oplus 5z_2)\,(1\oplus 8z_1\oplus 8z_2)
	\right).
	\]
	But
	\[
	(1\oplus 10z_1\oplus 4z_2)(1\oplus 9z_2)\ \ge\ (10z_1)(9z_2)=90\,z_1z_2\ \ge\ 6\,z_1z_2=(2z_1)(3z_2),
	\] 
and therefore
	\[
	\prod_{i=1}^4 (1\oplus a_{ii}z_1\oplus b_{ii}z_2)
	\ \ge\
	(2z_1)(3z_2)\,(1\oplus 7z_1\oplus 5z_2)\,(1\oplus 8z_1\oplus 8z_2),
	\] 
for all $z_1,z_2\ge 0$. Hence $\text{tdet}(M(z))=\prod_{i=1}^4 (1\oplus a_{ii}z_1\oplus b_{ii}z_2)$, and so
	\[
	P_{A,B}(z)=\text{tdet}(I\oplus z_1A\oplus z_2B)=\prod_{i=1}^4 (1\oplus a_{ii}z_1\oplus b_{ii}z_2),
	\] 
which illustrates Theorem \ref{nasc-linear factors}.
\end{example}

\begin{definition} \label{diagonally dominant}
A pair of triangularizable matrices $A,B \in M_n(\mathbb{R}_+)$ are said to be \textit{diagonally dominant}, if there 
exists a permutation matrix $P \in GL_n(\mathbb{R}_+)$ such that $A':= P^{-1}AP, \  B':= P^{-1}BP$ 
are row diagonal dominant for every row; that is, 
$\max\{a'_{ij}, b'_{ij}\} \ \leq\ \max\{a'_{ii}, b'_{ii}\}, \ \forall \ \, i, j$.
\end{definition}

We end the paper with the following corollary, whose proof follows from Theorem 
\ref{nasc-linear factors}.

\begin{corollary}\label{diagonal dominance-cor}
Let $A, B \in M_n(\mathbb{R}_+)$ be triangularizable matrices that are diagonally dominant. 
Then their characteristic polynomial $P_{A,B}(z)$ is a product of $n-$linear factors.  
\end{corollary}

\subsection{Algorithmic implications}\label{subsec:algorithms}\hspace*{0.5cm}

Theorems \ref{traingle} and \ref{first result} reduce (simultaneous) triangularization questions in max-algebras to 
standard directed graph theory problems. In particular, a triangularizing generalized permutation matrix can be 
constructed explicitly from a topological ordering of the associated digraph. In this section, we provide an algorithmic 
approach to determine whether a given max-algebraic matrix is triangularizable, or whether a pair of matrices is 
simultaneously triangularizable, and, in both cases, give the permutation matrix that puts these matrices to their respective triangular forms. As a consequence of Theorem \ref{nasc-linear factors}, we also provide an algorithm to determine whether 
the characteristic polynomial of a pair of matrices is linearly factorizable, and if so, express it as a product of linear factors. 

\smallskip

\noindent\textbf{Algorithm 1} : Triangularizability test and construction of a triangularizer.

\smallskip
\noindent\textbf{Input:} $A=(a_{ij}) \in M_n(\mathbb{R}_+)$.

\noindent\textbf{Output:} Either (i) \emph{Not triangularizable}, or (ii) a permutation
$v_1,\dots,v_n$ and a permutation matrix $P$ such that $P^{-1}AP$ is upper triangular.

\smallskip
\noindent\textbf{Procedure:}
\begin{enumerate}
\item Construct the digraph $G_A$ on $\{1,\dots,n\}$ with an edge $i\to j$ whenever
$a_{ij}>0$ and $i\neq j$.
\item Run a topological ordering routine on $G_A$:
\begin{enumerate}
\item Compute $\deg^{-}(v)$ (the indegree) of each vertex $v$.
\item Initialize a queue (or stack) $Q$ with all vertices of indegree $0$.
\item Initialize an empty list $L$.
\item While $Q$ is nonempty, remove a vertex $u$ from $Q$, append $u$ to $L$, and for each
outgoing edge $u \to w$, reduce $\deg^{-}(w)$ by $1$; if $\deg^{-}(w)=0$, insert $w$ into $Q$.
\end{enumerate}
\item If at the end $|L| < n$, then $G_A$ contains a directed cycle; return \emph{Not triangularizable}.
\item Otherwise, write $L=(v_1,\dots,v_n)$ and form the permutation matrix $P$ by
	\[
	P e_k = e_{v_k}\qquad (k=1,\dots,n);
	\]
equivalently, $P_{v_k,k} = 1$ for all $k$ and all other entries are $0$. 
\item Return $(v_1,\dots,v_n)$ and $P$.
\end{enumerate}

\smallskip
\noindent\textbf{Remark on complexity:} If $m_A$ is the number of nonzero off-diagonal entries of $A$,
then $G_A$ has $m_A$ edges except self-loops, and the above runs in $O(n+m_A)$ time.

\medskip

\noindent\textbf{Algorithm 2}: Simultaneous triangularizability test for a pair and construction.

\smallskip
\noindent\textbf{Input:} $A, B \in M_n(\mathbb{R}_+)$.

\noindent\textbf{Output:} Either (i) \emph{Not simultaneously triangularizable}, or (ii) a permutation
matrix $P$ such that both $P^{-1}AP$ and $P^{-1}BP$ are upper triangular.

\smallskip
\noindent\textbf{Procedure:}
\begin{enumerate}
\item First apply Algorithm 1 to $A$ and to $B$.
If either $A$ or $B$ is not triangularizable, then return \emph{Not simultaneously triangularizable}.
\item Construct the union digraph $G:=G_A\cup G_B$ on $\{1,\dots,n\}$ with an edge $i\to j$ whenever
$\max\{a_{ij},b_{ij}\} > 0$ and $i \neq  j$. (Equivalently, $G=G_{A\oplus B}$.)
\item Run a topological ordering routine on $G$ (as in Algorithm 1).
\item If $G$ has a directed cycle, return \emph{Not simultaneously triangularizable}.
\item Otherwise, if $L = (v_1,\dots,v_n)$ is a topological order, form $P$ by $P e_k=e_{v_k}$.
Then $P^{-1}AP$ and $P^{-1}BP$ are both upper triangular. 
\item Return $P$.
\end{enumerate}

\smallskip
\noindent\textbf{Remark on complexity:} If $m_A,m_B$ are the numbers of nonzero off-diagonal entries of
$A$ and $B$, then $G$ has at most $m_A+m_B$ edges, except self-loops, and the above runs in $O(n+m_A+m_B)$ time.

\medskip

\noindent\textbf{Algorithm 3}: Test for linear factorization in Theorem \ref{nasc-linear factors} and compute $P_{A,B}$ 
when factorizable.

\smallskip
\noindent\textbf{Input:} $A, B \in M_n(\mathbb{R}_+)$.

\noindent\textbf{Output:} Either (i) \emph{Not linearly factorizable}, or (ii) the characteristic
polynomial $P_{A,B}(z)$ in the explicit linear factor form
\[
P_{A,B}(z)=\prod_{i=1}^n (1\oplus \alpha_i z_1 \oplus \beta_i z_2),
\]
together with the coefficients $(\alpha_i,\beta_i)$.

\smallskip
\noindent\textbf{Procedure:}
\begin{enumerate}
\item Form $D:= A\oplus B = (d_{ij})$, where $d_{ij}:= a_{ij} \oplus b_{ij} = \max\{a_{ij}, b_{ij}\}$.
\item Compute
	\[
	T:= \text{tdet}(D)=\max_{\sigma\in S_n}\ \prod_{i=1}^n d_{i\sigma(i)}.
	\]
	\item Compute the diagonal product
	\[
	D_0:= \prod_{i=1}^n d_{ii} = \prod_{i=1}^n (a_{ii}\oplus b_{ii}).
	\]
\item If $T\neq D_0$, return \emph{Not linearly factorizable}. (In this case,
$P_{A,B}(z)$ is still defined by $P_{A,B}(z) = \text{tdet}(I\oplus z_1A\oplus z_2B)$, but we do not attempt
to expand it.)
\item If $T = D_0$, then by Theorem \ref{nasc-linear factors} the characteristic polynomial is a product of $n$ linear
factors. Set
	\[
	\alpha_i:=a_{ii},\qquad \beta_i:=b_{ii}\qquad (i=1,\dots,n),
	\]
	and output
	\[
	P_{A,B}(z)=\text{tdet}(I\oplus z_1A\oplus z_2B)=\prod_{i=1}^n (1\oplus \alpha_i z_1 \oplus \beta_i z_2)
	=\prod_{i=1}^n (1\oplus a_{ii}z_1 \oplus b_{ii}z_2).
	\]
\end{enumerate}

\smallskip
\noindent\textbf{Remark on complexity:} In the dense model, Step 2 can be carried out in $O(n^3)$ time
by standard assignment algorithms; sparse variants can be faster. When $T = D_0$, the algorithm outputs
$P_{A,B}(z)$ in factored form. Expanding the product may produce up to $3^n$ terms and is not needed
for the factorization criterion.

\smallskip
\noindent
As a final remark, we wish to point out that the results presented above carry over to min-algebras as well.

\section{Concluding remarks}

We summarize the main results obtained in this paper.

\begin{itemize}
\item A concrete criteria for triangularizability and simultaneous triangularizability of matrices in terms of acyclicity 
of the associated digraph is brought out. An example that distinguishes the classical and max-algebra contexts justifies 
this characterization.
\item Prominent results on simultaneous triangularization in classical linear algebra literature are generalized to 
max-algebras; the proofs involve nice graph theoretic techniques. These results bring out the similarities between these two 
areas.
\item The notion of unicellular matrices is introduced in max-algebra context, along with an example. To the best of 
our knowledge, this seems new in the context of max-algebras.
\item The notion of a characteristic polynomial of a pair of max-algebraic matrices in terms of the tropical determinant, 
along with an example from optimization to motivate the same, is introduced. Linear factorizations of this polynomial and 
its relationship with simultaneous triangularization in max-algebras are brought out.
\item Algorithms for triangularizability of a matrix, simultaneous triangularizability of a pair of matrices as well as 
linear factorizations of the tropical determinant of a pair of matrices, all in the max-algebra context, are brought out.
\item Nontrivial examples illustrating the results obtained are presented.
\item In the one variable case, the eigencone is generated by the columns of the Kleene star. In the two variable case, 
it is not clear at this stage whether the Kleene star provides all eigenvectors defined in Definition \ref{def:eigenvector-kernel}.  The authors therefore pose the problem of determining all eigenvectors of a pair $(A, B)$ of matrices and of characterizing 
the corresponding eigencone in this setting.
\item Finally, possible applications to joint spectral radius and periodic points are being investigated and we hope to 
consolidate the same soon.
\end{itemize}

\medskip
\noindent
\textbf{Achnowledgements:} The authors are grateful to the anonymous referee for a meticulous reading of the paper 
and for numerous comments and suggestions that has improved the paper. Askar Ali also thanks School of Mathematics, 
IISER Thiruvananthapuram, as part of this work was done during his tenure there.
 
\medskip
\noindent
{\bf Declarations:} 
\begin{itemize}
	\item All the authors have equal contributions in this work and declare that there is no conflict of interest.
	\item No funding was obtained or used for this work.
	\item No data were used in this work.
\end{itemize}

\bibliographystyle{amsplain}

\end{document}